\documentclass[11pt]{amsart}
\usepackage[english]{babel}
\usepackage{graphicx}
\usepackage[T1]{fontenc}
\usepackage{color}
\usepackage{bigints}

\newtheorem{Thm}{Theorem}
\newtheorem{Prop}[Thm]{Proposition}

\newtheorem{Lem}{Lemma}
\newtheorem{Rem}{Remark}
\newtheorem*{Prop*}{Proposition}
\newtheorem*{Cor*}{Corollary}
\newtheorem*{Thm*}{Theorem}

\def\xd {\mathrm{d}}

\def\eps{\varepsilon}
\def\R{{\mathbb R}}

\def\S{{\mathbb S}}

\DeclareMathAlphabet{\mathitbf}{OML}{cmm}{b}{it}

\newcommand{\dsp}{\displaystyle}

\newcommand{\dij}{\textrm{\DH}_{ij}}
\newcommand{\dji}{\textrm{\DH}_{ji}}

\begin{document}
\title[Maxwell-Stefan model in non-isothermal setting]{Maxwell-Stefan diffusion asymptotic for gas mixtures \\ in non-isothermal setting}
\bibliographystyle{plain}

\author[Harsha Hutridurga]{Harsha Hutridurga}
\address{H.H.: Department of Mathematics, Imperial College London, London, SW7 2AZ, United Kingdom.}
\email{h.hutridurga-ramaiah@imperial.ac.uk}

\author[Francesco Salvarani]{Francesco Salvarani}
\address{F.S.: Universit\'e Paris-Dauphine, Ceremade, UMR CNRS 7534, F-75775
Paris Cedex 16, France \& Universit\`a degli Studi di Pavia, Dipartimento di
Matematica, I-27100 Pavia, Italy} 
\email{francesco.salvarani@unipv.it}

\begin{abstract}
A mathematical model is proposed where the classical Maxwell-Stefan diffusion model for gas mixtures is coupled to an advection-type equation for the temperature of the physical system. This coupled system is derived from first principles in the sense that the starting point of our analysis is a system of Boltzmann equations for gaseous mixtures. We perform an asymptotic analysis on the Boltzmann model under diffuse scaling to arrive at the proposed coupled system.
\end{abstract}

\maketitle

\section{Introduction}\label{sec:introduction}

The Maxwell-Stefan theory \cite{max1866, ste1871} has been the most successful approach for describing diffusive phenomena in gaseous mixtures, and it is now the reference model for studying multicomponent diffusion. The Maxwell-Stefan system is a coupled system of cross-diffusion equations and it is commonly used in many scientific fields, for e.g., in engineering \cite{kri-wes-97} and in medical sciences \cite{bou-goe-gre, thi}. 

Despite its current utility, the mathematical studies on the subject are however quite recent (see \cite{gio, ern-gio2, ern-gio, gio_book}). In particular, existence and uniqueness issues, as well as the long-time behaviour, have been considered in \cite{bot-11, bou-gre-sal-12, jun-ste-13, che-jun-15}, whereas \cite{mcl-bou-14} deals with the numerical study of the Maxwell-Stefan equations.

In \cite{bou-gre-sal-15}, the authors provide the formal derivation of the Maxwell-Stefan diffusion equations starting from the non-reactive elastic Boltzmann system for monatomic gaseous mixtures \cite{bourgat1994microreversible, des-mon-sal, bou-gre-pav-sal}. They show that the zeroth and first order moments of appropriate solutions of the Boltzmann system, in the diffusive scaling and for vanishing Mach and Knudsen numbers limit, formally converge to the solution of the Maxwell-Stefan equations. This result, which lies in the research line introduced by Bardos, Golse and Levermore in \cite{bar-gol-lev-89, bar-gol-lev-91, bar-gol-lev-93}, has been obtained in the framework of Maxwellian cross sections.
Subsequently, the approach of \cite{bou-gre-sal-15} has been generalized in \cite{bou-gre-pav-16}, where the Maxwell-Stefan diffusion coefficients have been written in terms of explicit formulas with respect to the cross-sections,
and in \cite{hur-sal-16} where the explicit dependence of the Maxwell-Stefan binary diffusion coefficients with respect to the temperature of the mixture has been obtained for general analytical cross sections satisfying Grad's cutoff assumption \cite{gra-book}. 

All the previous results have been obtained in the isothermal case. However, as pointed out by Krishna and Wesselingh, ``perfectly isothermal systems are rare in chemical engineering practice and many processes such as distillation, absorption, condensation, evaporation and drying involve the simultaneous transfer of mass and energy across phase interfaces'' \cite[p.876]{kri-wes-97}. 

For this reason, it is natural to extend the strategy of \cite{bou-gre-sal-15} to the non-isothermal case, and this is the purpose of the present article: we provide here the asymptotics of the Boltzmann system for monatomic mixtures that leads to a non-isothermal form of the Maxwell-Stefan equations, and thus we can take into account the thermal diffusion contribution to the molar fluxes (thermophoresis).
We postulate that the solution of the Boltzmann system keeps the structure of a local Maxwellian and then we deduce, in the standard diffusive limit, the coupled relationships satisfied by the densities, the fluxes and the temperature at the macroscopic level which guarantee that the local Maxwellian structure is preserved by the time evolution of the system.

A major question is posed by the closure relationship. 
Indeed, as in the case of the Maxwell-Stefan system, the resulting equations for the densities, the fluxes and the temperature are, in the diffusive limit, 
 linearly dependent and an additional equation between the unknowns is necessary in order to close it.
As it is well known, in the isothermal Maxwell-Stefan system, the closure relationship consists in supposing that the sum of all molar fluxes $J_i$ is locally identically zero. This supplementary equation 
could be incompatible with some experimental behaviours in the non-isothermal setting: as pointed out in \cite{kri-wes-97}, indeed, in chemical vapour deposition (CVD) processes, thermal diffusion causes large, heavy gas molecules (for e.g., WF$_6$) to concentrate in cold regions whereas small, light molecules (such as H$_2$) to concentrate in hot regions. Hence, non-isothermal systems could require new closure relationships which, of course, relax to the isothermal one when the temperature is uniform in time and constant in space.

The closure relation that we suggest is the following: sum of the molar fluxes $J_i$ is locally proportional to the gradient of the total molar concentration, i.e.,
\[
\sum_{i=1}^n J_i = -\alpha \nabla c_{\rm tot}.
\]
With respect to the above mentioned closure relation, we characterize the total molar concentration $c_{\rm tot}$ and the temperature field $T(t,x)$ as solutions to a coupled system of evolution equations. Furthermore, the temperature-dependent flux-gradient relations derived in this paper -- see second line of \eqref{eq:coupled-system-not-closed} -- implies that the product $c_{\rm tot}T$ is space-independent. Hence the above mentioned closure relation postulated in this paper recovers the standard closure relation -- sum of the molar fluxes $J_i$ being locally identically zero -- in the isothermal case.

The outline of the paper is as follows: In subsection \ref{ssec:kinetic}, we introduce the kinetic model -- system of Boltzmann equations for gas mixtures -- and present the assumptions made on the Boltzmann collision kernels (Maxwellian molecules). Subsection \ref{ssec:diffuse-scale} deals with the scaling considered in this work and the main assumption made on the solutions to the scaled mesoscopic kinetic model. In subsection \ref{ssec:balance-laws} we derive the balance laws (mass, momentum and energy) -- see Proposition \ref{prop:balance-laws}. Emphasis is given on computing the coefficients in the balance laws -- given in terms of the velocity averages of certain statistical quantities. A formal asymptotic analysis (in the mean free path going to zero limit) is performed in subsection \ref{ssec:asymptotic} which culminates in Theorem \ref{thm:formal-eps}. Subsection \ref{ssec:closure} deals with the closure relation. Finally, in subsection \ref{ssec:quality}, we derive some qualitative properties on the total concentration $c_{\rm tot}(t,x)$ and the temperature field $T(t,x)$.

\section{Kinetic model and asymptotics}\label{sec:kinetic}

\subsection{Kinetic model}\label{ssec:kinetic}

The starting point of our analysis is a system of Boltzmann-type equations that models the evolution of a mixture of ideal monatomic inert gases $\mathcal{A}_i$, $i=1,\dots, n$ with $n\ge2$, subject to elastic mechanical collisions between each other. More precisely, for the unknown probability density functions $f_i(t,x,v)\ge0$, we consider the Cauchy problem
\begin{align}
\partial_t f_i+v \cdot \nabla_x f_i
& =  \sum_{j=1}^n Q_{ij}(f_i,f_j) 
\quad & \mbox{ for }(t,x,v)\in (0,\infty)\times\R^3\times\R^3, \label{eq:BE_nonscaled}
\\
f_i(0,x,v)
& = f^{\rm in}_i (x,v)
\quad  & \mbox{ for }(x,v)\in \R^3\times\R^3, \label{eq:BE_nonscaled_initial}
\end{align}
for each $i=1,\dots,n$ where $Q_{ij}(\cdot,\cdot)$ denotes the bilinear integral operator describing the collisions of molecules of species $\mathcal A_i$ with molecules of species $\mathcal A_j$. In the above model, we have supposed that there are no external forces acting on the gas mixture. Hence any given particle travels in a straight line (ballistic motion) until it encounters another particle resulting in a mechanical collision which is assumed here to be elastic. To facilitate the definition of the collision operator $Q_{ij}(\cdot,\cdot)$, consider two particles belonging to the species $\mathcal A_i$ and $\mathcal A_j$, $1\le i,j\le n$, with masses $m_i$, $m_j$, and pre-collisional velocities $v'$, $v_*'$. A microscopic collision is an instantaneous phenomenon which modifies the velocities of the particles, which become $v$ and $v_*$, obtained by imposing the conservation of both momentum and kinetic energy:
\begin{equation} \label{eq:coll_momen_energ}
m_i v'+m_j v'_* = m_i v+m_j v_*,\qquad 
\frac 12 m_i\,|v'|^2 + \frac 12 m_j\,|v_*'|^2 = \frac 12 m_i\,|v|^2 + \frac 12
m_j\,|v_*|^2.
\end{equation}
The previous equations allow us to write $v'$ and $v'_*$ in terms of $v$ and $v_*$:
\begin{equation} \label{eq:v*}
v'=\frac{1}{m_i +m_j}(m_i v+m_j v_* + m_j \vert v- v_*\vert\, \sigma), \qquad 
v'_*= \frac{1}{m_i+m_j}(m_i v+m_j v_* - m_i \vert v- v_*\vert\, \sigma),
\end{equation}
where $\sigma\in\S^2$ describes the two degrees of freedom in \eqref{eq:coll_momen_energ}.

If $f$ and $g$ are nonnegative functions, the operator describing the collisions between molecules of species $\mathcal A_i$ and molecules of species $\mathcal A_j$ is defined by
\begin{equation} \label{eq:Q_bi}
Q_{ij}(f,g)(v)
:= 
\iint\limits_{\R^3\times\S^2} B_{ij}(v,v_*,\sigma)\Big[ f(v')g(v'_*) -
f(v)g(v_*) \Big]  \,\xd\sigma\, \xd v_*
\end{equation}
where $v'$ and $v'_*,$ are given by the relation \eqref{eq:v*}, and the cross sections $B_{ij}$ satisfy the microreversibility assumptions: $B_{ij}(v,v_*,\sigma)=B_{ji}(v_*,v, \sigma)$ and $B_{ij}(v,v_*,\sigma)=B_{ij}(v',v'_*,\sigma)$. 

The operators $Q_{ij}$ can be written in weak form. For example, by using the changes of variables $(v,v_*) \mapsto (v_*,v)$ and $(v,v_*) \mapsto (v',v'_*)$, we have
\begin{multline} \label{eq:weak_bi_utile}
\int\limits_{\R^3}Q_{ij}(f,g)(v)\,\psi(v)\, \xd v\\
= -\frac 12 \iiint\limits_{\R^6\times\S^2} B_{ij}(v,v_*, \sigma) \Big[ f(v')g(v'_*) -
f(v)g(v_*) \Big] \Big[ \psi(v') -\psi(v) \Big] \, \xd\sigma\,\xd v\, \xd v_*\\ 
=\iiint\limits_{\R^6\times\S^2} B_{ij}(v,v_*, \sigma)\,f(v)g(v_*)\, 
\Big[ \psi(v') -\psi(v) \Big] \, \xd\sigma\,\xd v\, \xd v_*,
\end{multline}
or
\begin{multline} \label{eq:weak_el_bi} 
\int\limits_{\R^3}Q_{ij}(f,g)(v)\,\psi(v)\, \xd v + \int\limits_{\R^3}Q_{ji}(g,f)(v)\, \phi(v)\, \xd v\\
=-\frac 12 \iiint\limits_{\R^6\times\S^2}B_{ij}(v,v_*,\sigma) \Big[ f(v')g(v'_*) -
f(v)g(v_*) \Big ] \Big[\psi(v') +
\phi(v'_*)-\psi(v)-\phi(v_*)\Big]\,\xd\sigma\,\xd v\, \xd v_*,
\end{multline}
for any $\psi$, $\phi: \, \R^3\to \R$ such that the integrals on the left hand sides of \eqref{eq:weak_bi_utile} and \eqref{eq:weak_el_bi} are well defined. 

The collision kernels $B_{ij}$ only depend on the modulus of the relative velocity and on the cosine of the deviation angle, i.e.,
\begin{align*}
B_{ij}(v,v_*,\sigma) = B_{ij}\left( |v-v_*|,\cos\theta \right)
\quad
\mbox{ with }
\quad
\cos\theta=\frac{v-v_*}{|v-v_*|}\cdot\sigma.
\end{align*}
In order to ease the presentation and also to ensure that the resulting mathematical model for the molar concentrations be simple, throughout this work we stick to the case of Maxwellian molecules. More specifically, we shall work with collision kernels that are independent of the relative velocity, i.e., they are of the form
\begin{align}\label{eq:B-ij}
B_{ij}(v,v_*,\sigma) = b_{ij}(\cos\theta),
\end{align}
where we assume that the angular collision kernels $b_{ij}\in L^1(-1,+1)$ and are even. Observe that, because of the microreversibility assumption on the collision kernels, we have
\begin{align*}
B_{ij}(v,v_*,\sigma)=B_{ji}(v_*,v,\sigma)\implies b_{ij}\left(\frac{v-v_*}{|v-v_*|}\cdot\sigma\right) = b_{ji}\left(\frac{v_*-v}{|v-v_*|}\cdot\sigma\right),
\end{align*}
and that, by parity, $b_{ij}(\cos\theta) = b_{ji}(\cos\theta)$.
\begin{Rem}\label{rem:conservation}
Taking $\psi(v) = 1$ in the weak form \eqref{eq:weak_bi_utile}, yields
\begin{equation}\label{eq:consv}
\int\limits_{\R^3} Q_{ij}(f, g)(v) \, \xd v = 0
\qquad \mbox{ for all }i,j=1,\dots,n
\end{equation}
which helps us deduce the conservation of the total number of molecules of species $\mathcal A_i$. Moreover in \eqref{eq:weak_el_bi}, if $\psi(v) = m_i\,v$ and $\phi(v_*) = m_j\,v_*$, and then if
$\psi(v) =m_i\,|v|^2/2$ and $\phi(v)=m_j\,|v_*|^2/2$, we recover the conservation of the total momentum and of the total kinetic
energy during the collision between a particle of species $\mathcal A_i$ and a particle of species $\mathcal A_j$:
\begin{equation*}
\int\limits_{\R^3} Q_{ij}(f,g)(v) \, \left(
\begin{array}{c}
m_i\,v\\ m_i\,{|v|^2}/2
\end{array}
\right) \, \xd v
+ \, \int\limits_{\R^3}Q_{ji}(g,f)(v) \, \left(
\begin{array}{c}
m_j\,v \\ m_j\,{|v|^2}/2
\end{array}
\right) \, \xd v= 0.
\end{equation*}
\end{Rem}

\subsection{Diffuse scaling and main assumptions}\label{ssec:diffuse-scale}

In order to arrive at the diffusive limit, we introduce a scaling parameter $0<\eps\ll1$ which represents the mean free path. The space-time variables are scaled as $(t,x)\mapsto (\eps^2 t, \eps x)$. Note that the velocity variable is not scaled. The unknown distribution functions in the transformed variables are denoted by $f_i^{\eps}$. Each distribution function $f_i^\eps$ solves the following scaled version of \eqref{eq:BE_nonscaled}-\eqref{eq:BE_nonscaled_initial}:
\begin{align} 
\eps\, \partial_t f_i^{\eps}+v \cdot \nabla_x f_i^{\eps}
& =  \frac 1\eps
 \sum_{j=1}^n Q_{ij}(f_i^{\eps},f_j^{\eps}) 
\quad  & \mbox{ for }(t,x,v)\in (0,\infty)\times\R^3\times\R^3, \label{eq:BE_scaled}\\
f^\eps_i(0,x,v) & = 
\left(
f^{\rm in}_i
\right)^\eps  
 (x,v)
\quad & \mbox{ for }(x,v)\in \R^3\times\R^3. \label{eq:BE_scaled_initial}
\end{align}
The initial data are assumed to be such that the associated local macroscopic velocities are of $\mathcal{O}(\eps)$, i.e.,
\[
\int\limits_{\R^3}
v
\left(
f^{\rm in}_i
\right)^\eps  
 (x,v)
\, \xd v
= \eps\, c^{\rm in}_i(x) u^{\rm in}_i(x) 
\]
for some $c^{\rm in}_i:\R^3\to[0,\infty)$ and $u^{\rm in}_i:\R^3\to\R^3$.

The main assumption in our work is that the evolution following \eqref{eq:BE_scaled} keeps the distribution functions $f^\eps_i(t,x,v)$ in local Maxwellian states. We hence suppose that there exist the local densities $c^\eps_i: [0,\infty)\times \R^3 \to [0,\infty)$, the local macroscopic velocities $u^\eps_i: [0,\infty)\times\R^3 \to \R^3$ and the local temperatures $T^\eps_i: [0,\infty)\times\R^3 \to [0,\infty)$ such that the solution to the scaled Boltzmann system \eqref{eq:BE_scaled} has the following Maxwellian structure:
\begin{align}\label{eq:Maxwellian}
f_i^\eps (t, x, v)
= c^\eps_i(t,x)
\left(
\frac{m_i}{2\pi k T^\eps_i(t,x)}
\right)^{3/2}
e^{-m_i |v - \eps u^\eps_i(t,x)|^2/2 k T^\eps_i(t,x)}
\end{align} 
where $k$ is the Boltzmann constant. We record the zeroth, first and second moments for the above Maxwellian states:
\begin{align}\label{eq:moments_012}
\bigintss\limits_{\R^3} f^\eps_i(t,x,v) 
\left(
\begin{array}{c}
1\\ v \\ |v|^2
\end{array}
\right) \, \xd v
=
\left(
\begin{array}{c}
 c^\eps_i(t,x)\\[0.1 cm]
\eps\, c^\eps_i(t,x) u^\eps_i(t,x)\\[0.1 cm]
\frac{3k}{m_i} c^\eps_i(t,x) T^\eps_i(t,x) + \eps^2 c^\eps_i(t,x) |u^\eps_i(t,x)|^2
\end{array}
\right).
\end{align}
\textsc{notation:} For $\ell\in\{1, 2, 3\}$, we denote by $w_{(\ell)}$ the $\ell$-th component of any vector $w\in\R^3$.\\
Note that we postulate the ansatz \eqref{eq:Maxwellian} for the distribution functions. This line of attack to address diffusion limit procedures in the context of Boltzmann models is borrowed from \cite{bou-gre-sal-15, hur-sal-16}.
Also, note that the choice of $\mathcal{O}(\eps)$ local macroscopic velocities in the ansatz \eqref{eq:Maxwellian} results in the first moment of the distribution functions to be of $\mathcal{O}(\eps)$ since we are only interested in the pure diffusive dynamics.

\subsection{Balance laws}\label{ssec:balance-laws}
To derive a macroscopic description out of the mesoscopic dynamics of the Boltzmann-type equations such as \eqref{eq:BE_scaled}, we need to arrive at balance equations obtained by integrating the transport model \eqref{eq:BE_scaled} with respect to the velocity variable $v$ only. The next result records various macroscopic equations associated with the scaled Boltzmann-like system  
\eqref{eq:BE_scaled}-\eqref{eq:BE_scaled_initial}.

\begin{Prop}\label{prop:balance-laws}
Suppose that the evolution according to \eqref{eq:BE_scaled} keeps the distribution functions $f^\eps_i(t,x,v)$ in the local Maxwellian states \eqref{eq:Maxwellian} for all time $t>0$. Then, the local macroscopic observables $(c^\eps_i, u^\eps_i, T^\eps)$ solve the following equations.
We have the mass balance equations:
\begin{align}\label{eq:eps-mass-bal}
\partial_t c^\eps_i 
+ \nabla_x \cdot (c^\eps_i u^\eps_i) = 0
\qquad 
\mbox{ for }(t,x) \in (0,\infty)\times\R^3,
\end{align}
for each $1\le i \le n$. We further have the momentum balance:
\begin{align}\label{eq:eps-momen-bal}
\eps^2 \Big( \partial_t\left( c^\eps_i u^\eps_i\right) + \nabla_x \cdot  \left( c^\eps_i u^\eps_i \otimes u^\eps_i \right) \Big)
+ \frac{k}{m_i} \nabla_x \left( c^\eps_i T^\eps_i \right)
= 
\Theta^\eps_i
\qquad 
\mbox{ for }(t,x) \in (0,\infty)\times\R^3,
\end{align}
for each $1\le i\le n$, where the right hand side of \eqref{eq:eps-momen-bal} reads
\begin{align}\label{eq:theta-eps}
\Theta^\eps_i(t,x)
=
\sum_{j\not=i}
2\pi \|b_{ij}\|_{L^1}
\frac{m_j}{m_i + m_j}
\Big(
c^\eps_i c^\eps_j u^\eps_j 
-
c^\eps_j c^\eps_i u^\eps_i
\Big)
+ \mathcal{O}(\eps).
\end{align}
Furthermore, the energy balance reads
\begin{align}\label{eq:eps-energy-bal}
&
\eps \left(
\frac{3k}{m_i}
\partial_t  \left( c^\eps_i T^\eps_i \right)
+ \frac{5k}{m_i} \nabla_x \cdot  \left( c^\eps_i T^\eps_i u^\eps_i \right) \right)
+ \eps^3 \left( \partial_t  \left( c^\eps_i |u^\eps_i|^2 \right)
+ \frac{3k}{m_i} \nabla_x \cdot  \left( c^\eps_i T^\eps_i |u^\eps_i|^2 u^\eps_i \right) \right)
= \Xi^\eps_i
\end{align}
for $(t,x) \in (0,\infty)\times\R^3$ and for each $1\le i\le n$, where the right hand side of \eqref{eq:eps-energy-bal} reads
\begin{equation}\label{eq:xi-eps}
\begin{aligned}
\Xi^\eps_i (t,x) 
& =
\frac{1}{\eps} \sum_{j\not=i}
\|b_{ij}\|_{L^1}
\frac{m_j}{(m_i + m_j)^2}
c^\eps_i c^\eps_j \left(T^\eps_i-T^\eps_j\right)
\\
& + \eps\, \sum_{j\not=i}
2 \|b_{ij}\|_{L^1} c^\eps_i  c^\eps_j m_j
\left(
\frac{(m_j u^\eps_j + m_i u^\eps_i)\cdot (u^\eps_j - u^\eps_i)}{(m_i + m_j)^2}
\right).
\end{aligned}
\end{equation}
\end{Prop}

\begin{proof}
To arrive at the balance equations \eqref{eq:eps-mass-bal}-\eqref{eq:eps-momen-bal}-\eqref{eq:eps-energy-bal}, multiply the scaled equation \eqref{eq:BE_scaled} by $(1,v_{(\ell)}, |v|^2)$ and integrate over all possible velocities in $\R^3$ yielding
\begin{align}\label{eq:balance-comp-1}
\eps\, \partial_t \bigintss\limits_{\R^3} f^\eps_i(t,x,v) 
\left(
\begin{array}{c}
1\\ v_{(\ell)} \\ |v|^2
\end{array}
\right) \, \xd v
+ \nabla_x \cdot 
\bigintss\limits_{\R^3} v f^\eps_i(t,x,v) 
\left(
\begin{array}{c}
1\\ v_{(\ell)} \\ |v|^2
\end{array}
\right) \, \xd v
=
\frac{1}{\eps}
\sum_{j=1}^n
\bigintss\limits_{\R^3} Q_{ij}(f^\eps_i, f^\eps_j)
\left(
\begin{array}{c}
1\\ v_{(\ell)} \\ |v|^2
\end{array}
\right) \, \xd v.
\end{align}
The assumption of Maxwellian structure \eqref{eq:Maxwellian} on the solution $f^\eps_i(t,x,v)$ helps us compute the divergence term in the second line on the left hand side of the above equation:
\begin{equation}\label{eq:div-vell}
\begin{aligned}
\nabla_x \cdot & \left( \int\limits_{\R^3} v_{(\ell)} f^\eps_i(t,x,v) v\, \xd v \right)
\\
& =
\frac{k}{m_i}
\frac{\partial}{\partial x_{(\ell)}}
\Big(
c^\eps_i(t,x)
T^\eps_i(t,x)
\Big)
+ \eps^2
\sum_{k=1}^3
\frac{\partial}{\partial x_{(k)}}
\Big(
c^\eps_i (t,x)
\left( u^\eps_i \right)_{(k)} (t,x)
\left( u^\eps_i \right)_{(\ell)} (t,x)
\Big).
\end{aligned}
\end{equation}
Next, we have for the divergence term in the third line on the left hand side of \eqref{eq:balance-comp-1}:
\begin{equation*}
\begin{aligned}
& \nabla_x
\cdot 
\left(
\int\limits_{\R^3}
|v|^2
v f^\eps_i(v)
\, {\rm d}v
\right)
=
\sum_{k=1}^3
\frac{\partial}{\partial x_{(k)}}
\int\limits_{\R^3}
\left(
|v_{(1)}|^2
+ |v_{(2)}|^2
+ |v_{(3)}|^2
\right)
v_{(k)}
f^\eps_i(v)
\, {\rm d}v
\\
& =
\sum_{k=1}^3
\frac{\partial}{\partial x_{(k)}}
\int\limits_{\R^3}
\left(
|v_{(1)} + \eps \left( u^\eps_i \right)_{(1)} |^2
+ |v_{(2)}+ \eps \left( u^\eps_i \right)_{(2)}|^2
+ |v_{(3)}+ \eps \left( u^\eps_i \right)_{(3)}|^2
\right) \times
\\
&
\hspace{3.0 cm} \left( v_{(k)} + \eps \left( u^\eps_i \right)_{(k)} \right)
c^\eps_i (t,x)
\left(
\frac{m_i}{2\pi k T^\eps_i(t,x)}
\right)^{3/2}
e^{\frac{-m_i |v|^2}{2 k T^\eps_i(t,x)}}
\, {\rm d}v.
\end{aligned}
\end{equation*}
Therefore, we have:
\begin{align}\label{eq:div-v2}
\nabla_x
\cdot 
\left(
\int\limits_{\R^3}
|v|^2
v f^\eps_i(v)
\, {\rm d}v
\right)
= \eps 
\frac{5k}{m_i}
\nabla_x \cdot 
\left(
c^\eps_i T^\eps_i u^\eps_i
\right)
+
\eps^3
\frac{3k}{m_i}
\nabla_x \cdot 
\left(
c^\eps_i T^\eps_i |u^\eps_i|^2 u^\eps_i
\right).
\end{align}
Observation \eqref{eq:consv} in Remark \ref{rem:conservation} implies that the first line on the right hand side of \eqref{eq:balance-comp-1} vanishes. Under the Maxwellian molecules assumption \eqref{eq:B-ij} and under the local Maxwellian states assumption \eqref{eq:Maxwellian} on the solution $f^\eps_i(t,x,v)$, the second line on the right hand side of \eqref{eq:balance-comp-1} has already been computed by L. Boudin, B. Grec and F. Salvarani \cite[Section~4]{bou-gre-sal-15}. We will simply borrow the end result of their computation below:
\begin{align}\label{eq:line2-rhs}
\frac{1}{\eps} \sum_{j=1}^n \int\limits_{\R^3} v_{(\ell)} Q_{ij}(f^\eps_i, f^\eps_j)(v)\, \xd v
=
\sum_{j\not=i}
\frac{2\pi m_j \|b_{ij}\|_{L^1}}{m_i + m_j}
\Big(
c^\eps_i c^\eps_j \left( u^\eps_j \right)_{(\ell)}
-
c^\eps_j c^\eps_i \left( u^\eps_i \right)_{(\ell)}
\Big)
+ \mathcal{O}(\eps).
\end{align}
On the other hand to arrive at the expression \eqref{eq:xi-eps}, let us consider the right hand side of the third line in \eqref{eq:balance-comp-1}:
\begin{align*}
\frac{1}{\eps} \sum_{j\not=i} \int\limits_{\R^3} |v|^2 Q_{ij}(f^\eps_i, f^\eps_j) \, \xd v
& =
\frac{1}{\eps}
\sum_{j\not=i}\,
\iiint\limits_{\R^6\times\S^2}
B_{ij}(v,v_*,\sigma)
f^\eps_i(v)
f^\eps_j(v_*)
\left(
|v'|^2 - |v|^2
\right)
\, {\rm d}\sigma\, {\rm d}v\, {\rm d}v_*
\\
& =
\frac{1}{\eps}
\sum_{j\not=i} 
\sum_{\ell=1}^3\, 
\iiint\limits_{\R^6\times\S^2}
B_{ij}(v,v_*,\sigma)
f^\eps_i(v)
f^\eps_j(v_*)
\left(
|v'_{(\ell)}|^2 
- |v_{(\ell)}|^2 
\right)
\, {\rm d}\sigma\, {\rm d}v\, {\rm d}v_*
\end{align*}
where we have used the weak form \eqref{eq:weak_bi_utile} with $\psi(v)=|v|^2$.
Next, using the relation \eqref{eq:v*}, the above expression can be further simplified as
\begin{align*}
\frac{1}{\eps} \sum_{j\not=i} \int\limits_{\R^3} |v|^2 Q_{ij}(f^\eps_i, f^\eps_j) \, \xd v
& =
\frac{1}{\eps} \sum_{j\not=i} \sum_{\ell=1}^3 \|b_{ij}\|_{L^1}
\left(
\frac{m^2_i}{(m_i + m_j)^2}
- 1
\right)
\iint\limits_{\R^6} 
f^\eps_i(v)
f^\eps_j(v_*)
|v_{(\ell)}|^2
\, {\rm d}v \, {\rm d}v_*
\\
& +
\frac{1}{\eps} \sum_{j\not=i} \sum_{\ell=1}^3 \|b_{ij}\|_{L^1}
\frac{m^2_j}{(m_i + m_j)^2}
\iint\limits_{\R^6} 
f^\eps_i(v)
f^\eps_j(v_*)
|v_{*(\ell)}|^2
\, {\rm d}v \, {\rm d}v_*
\\
& +
\frac{1}{\eps} \sum_{j\not=i} \sum_{\ell=1}^3 \|b_{ij}\|_{L^1}
\frac{2m_i m_j}{(m_i + m_j)^2}
\iint\limits_{\R^6} 
f^\eps_i(v)
f^\eps_j(v_*)
v_{(\ell)} v_{*(\ell)}
\, {\rm d}v \, {\rm d}v_*
\end{align*}
\begin{align*}
\hspace{3.0 cm} &
 +
\frac{1}{\eps} \sum_{j\not=i} \sum_{\ell=1}^3
\frac{2}{(m_i + m_j)^2}
\iiint\limits_{\R^6\times\S^2}
B_{ij}(v,v_*,\sigma)
f^\eps_i(v)
f^\eps_j(v_*)
\times
\\
&
\hspace{2.0 cm}\left(
m_i m_j |v-v_*| v_{(\ell)}
+
m^2_j |v-v_*| v_{*(\ell)}
\right)
\sigma_{(\ell)}
\, {\rm d}\sigma\, {\rm d}v\, {\rm d}v_*
\\
& + 
\frac{1}{\eps} \sum_{j\not=i} \sum_{\ell=1}^3
\frac{m^2_j}{(m_i + m_j)^2}
\iiint\limits_{\R^6\times\S^2}
B_{ij}(v,v_*,\sigma)
f^\eps_i(v)
f^\eps_j(v_*)
|v-v_*|^2 |\sigma_{(\ell)}|^2
\, {\rm d}\sigma\, {\rm d}v\, {\rm d}v_*
\\
& =: 
\mathcal{I}_1 + \mathcal{I}_2 + \mathcal{I}_3 + \mathcal{I}_4 + \mathcal{I}_5.
\end{align*}
Substituting the local Maxwellian structure \eqref{eq:Maxwellian} for the distribution function $f^\eps_i(t,x,v)$ in the above integrals and computing the thus obtained Gaussian integrals yield
\begin{equation*}\label{eq:mathcal-I-123}
\begin{aligned}
\mathcal{I}_1 
& = \frac{1}{\eps} \sum_{j\neq i}
\left(
\|b_{ij}\|_{L^1}
\left(
\frac{m^2_i}{(m_i + m_j)^2}
- 1
\right)
c^\eps_i c^\eps_j
\frac{3kT^\eps_i}{m_i}
+
\eps^2
\|b_{ij}\|_{L^1}
\left(
\frac{m^2_i}{(m_i + m_j)^2}
- 1
\right)
c^\eps_i c^\eps_j
|u^\eps_i|^2
\right),
\\
\mathcal{I}_2 
& = \frac{1}{\eps} \sum_{j\neq i}
\left(
\|b_{ij}\|_{L^1}
\frac{m^2_j}{(m_i + m_j)^2}
c^\eps_i c^\eps_j
\frac{3kT^\eps_j}{m_j}
+
\eps^2
\|b_{ij}\|_{L^1}
\frac{m^2_j}{(m_i + m_j)^2}
c^\eps_i c^\eps_j
|u^\eps_j|^2
\right),
\\
\mathcal{I}_3
& = \frac{1}{\eps} \sum_{j\neq i}
\left(
\eps^2 
\|b_{ij}\|_{L^1}
\frac{2m_i m_j}{(m_i + m_j)^2}
c^\eps_i c^\eps_j
\left(
u^\eps_i
\cdot 
u^\eps_j
\right)
\right).
\end{aligned}
\end{equation*}
Now, to treat the integral $\mathcal{I}_4$, let us introduce the polar variable $\varphi\in [0,2\pi]$ so that we can find the relationships between the Euclidean coordinates of $\sigma$ and the spherical ones, namely 
\begin{align*}
\sigma_{(1)}=\sin \theta\cos\varphi, \quad \sigma_{(2)}=\sin
\theta\sin\varphi, \quad \sigma_{(3)}=\cos \theta.
\end{align*}
In the integral $\mathcal{I}_4$, note that the terms for $\ell=1$ or $2$  in the sum are zero because
\begin{align*}
\int\limits_0^{2\pi}\sin\varphi\,\xd\varphi=\int\limits_0^{2\pi}\cos\varphi\,\xd\varphi=0,
\end{align*}
and for $\ell=3$, because $b_{ij}$ is even, one has
\begin{align*}
\int\limits_{\S^2} b_{ij}\left(\frac{v-v_*}{|v-v_*|}\cdot\sigma\right)\sigma_{(3)} \,
\xd\sigma=2\pi\int\limits_0^\pi\sin\theta\cos\theta\,
b_{ij}(\cos\theta)\,\xd\theta=2\pi\int\limits_{-1}^1 \eta \,b_{ij}(\eta)\,\xd\eta=0.
\end{align*}
Hence the integral $\mathcal{I}_4$ vanishes. Next, we get to the computation of the integral $\mathcal{I}_5$. Before we go further, we make the following observation:
\begin{align*}
\sum_{\ell=1}^3
\int\limits_{\S^2}
B_{ij}(v,v_*,\sigma)
|\sigma_{(\ell)}|^2
\, {\rm d}\sigma
= \int\limits_{\S^2}
B_{ij}(v,v_*,\sigma)
\, {\rm d}\sigma =
\|b_{ij}\|_{L^1}
\end{align*}
because $|\sigma |^2=1$. Hence the integral $\mathcal{I}_5$ reduces to
\begin{align*}
\mathcal{I}_5 
=
\frac{1}{\eps} \sum_{j\not=i}
\|b_{ij}\|_{L^1}
\frac{m^2_j}{(m_i + m_j)^2}
\iint\limits_{\R^6}
f^\eps_i(v)
f^\eps_j(v_*)
|v-v_*|^2
\, {\rm d}v\, {\rm d}v_*.
\end{align*}
Substituting the local Maxwellian structures \eqref{eq:Maxwellian} for the distribution functions $f^\eps_i$ and performing the change of variables: $(v,v_*)\mapsto (v+\eps u^\eps_i, v_*+\eps u^\eps_j)$ yields
\begin{align*}
\mathcal{I}_5 
& =
\frac{1}{\eps} \sum_{j\not=i}
\|b_{ij}\|_{L^1}
\frac{m^2_j}{(m_i + m_j)^2}
c^\eps_i c^\eps_j
\left(
\frac{m_i}{2\pi k T^\eps_i}
\right)^{3/2}
\left(
\frac{m_j}{2\pi k T^\eps_j}
\right)^{3/2}\times
\\
& \hspace{2.0 cm}\iint\limits_{\R^6}
\left|v + \eps u^\eps_i - v_* - \eps u^\eps_j\right|^2
e^{-m_i |v|^2/2 k T^\eps_i}
e^{-m_j |v_*|^2/2 k T^\eps_j}
\, {\rm d}v \, {\rm d}v_*
\\
& =
\frac{1}{\eps} \sum_{j\not=i}
\|b_{ij}\|_{L^1}
\frac{m^2_j}{(m_i + m_j)^2}
c^\eps_i c^\eps_j
\left(
\frac{m_i}{2\pi k T^\eps_i}
\right)^{3/2}
\left(
\frac{m_j}{2\pi k T^\eps_j}
\right)^{3/2}\times
\\
&\hspace{0.8 cm}
\sum_{\ell=1}^3
\iint\limits_{\R^6}
\Big(
|v_{(\ell)}|^2 
+ 
\eps^2 |\left( u^\eps_i\right)_{(\ell)}|^2 
+
2 \eps v_{(\ell)} \left( u^\eps_i\right)_{(\ell)}
+
|v_{*(\ell)}|^2 
+
\eps^2 |\left( u^\eps_j\right)_{(\ell)}|^2 
+
2 \eps v_{*(\ell)} \left( u^\eps_j\right)_{(\ell)}
\\
&\hspace{1.0 cm}
- 2 v_{(\ell)} v_{*(\ell)}
-2 \eps v_{(\ell)} \left( u^\eps_j\right)_{(\ell)}
- 2 \eps v_{*(\ell)} \left( u^\eps_i\right)_{(\ell)}
-2 \eps^2 \left( u^\eps_i\right)_{(\ell)} \left( u^\eps_j\right)_{(\ell)}
\Big)
e^{-m_i |v|^2/2 k T^\eps_i}
e^{-m_j |v_*|^2/2 k T^\eps_j}
\, {\rm d}v \, {\rm d}v_*.
\end{align*}
Note that some of the Gaussian integrals in the above sum vanish. Thus, the integral $\mathcal{I}_5$ simplifies as follows:
\begin{align*}
\mathcal{I}_5
& =
\frac{1}{\eps} \sum_{j\not=i}
\|b_{ij}\|_{L^1}
\frac{m^2_j}{(m_i + m_j)^2}
c^\eps_i c^\eps_j
\left(
\frac{m_i}{2\pi k T^\eps_i}
\right)^{3/2}
\left(
\frac{m_j}{2\pi k T^\eps_j}
\right)^{3/2}
\sum_{\ell=1}^3
\iint\limits_{\R^6}
\Big(
|v_{(\ell)}|^2 
+
|v_{*(\ell)}|^2 
\\
&\hspace{2.0 cm}
+ 
\eps^2 |\left( u^\eps_i\right)_{(\ell)}|^2 
+
\eps^2 |\left( u^\eps_j\right)_{(\ell)}|^2 
-2
\eps^2 \left( u^\eps_i\right)_{(\ell)} \left( u^\eps_j\right)_{(\ell)}
\Big)
e^{-m_i |v|^2/2 k T^\eps_i}
e^{-m_j |v_*|^2/2 k T^\eps_j}
\, {\rm d}v \, {\rm d}v_*
\\
& =
\frac{1}{\eps} \sum_{j\not=i}
\|b_{ij}\|_{L^1}
\frac{m^2_j}{(m_i + m_j)^2}
c^\eps_i c^\eps_j
\left(
\frac{3kT^\eps_i}{m_i}
+
\frac{3kT^\eps_j}{m_j}
+
\eps^2 |u^\eps_i|^2 
+
\eps^2 |u^\eps_j|^2 
-2
\eps^2 u^\eps_i \cdot u^\eps_j
\right).
\end{align*}
Summing all the above integral computations together, i.e., $\mathcal{I}_1$ through $\mathcal{I}_5$, we have
\begin{equation}\label{eq:line3-rhs}
\begin{aligned}
\frac{1}{\eps} \sum_{j\not=i} \int\limits_{\R^3} |v|^2 Q_{ij}(f^\eps_i, f^\eps_j) \, \xd v
& =
\frac{1}{\eps} \sum_{j\not=i}
\|b_{ij}\|_{L^1}
\frac{m_j}{(m_i + m_j)^2}
c^\eps_i c^\eps_j \left(T^\eps_i-T^\eps_j\right)
\\
& + \eps\, \sum_{j\not=i}
2 \|b_{ij}\|_{L^1} c^\eps_i  c^\eps_j m_j
\left(
\frac{(m_j u^\eps_j + m_i u^\eps_i)\cdot (u^\eps_j - u^\eps_i)}{(m_i + m_j)^2}
\right).
\end{aligned}
\end{equation}
Finally, using the moments' computations \eqref{eq:moments_012}, the divergence terms \eqref{eq:div-vell}-\eqref{eq:div-v2} and the right hand side terms \eqref{eq:line2-rhs}-\eqref{eq:line3-rhs} in the balance equation \eqref{eq:balance-comp-1}, we have arrived at the result.
\end{proof}

In this article, we suppose that the cross sections are of Maxwellian type. Of course, it is possible to consider the problem under more general
assumptions on the cross sections. For example, in \cite{hur-sal-16} the authors considered collision kernels of the form
\begin{align*}
B_{ij}(v,v_*,\sigma) = b_{ij}\left(\cos(\theta)\right)\Phi(\vert v - v_*\vert)
\end{align*}
with the kinetic collision kernel having the structure (see \cite[section 4]{hur-sal-16} for precise details)
\begin{align*}
\Phi(\vert v - v_*\vert) = \sum_{n\in\mathbb{N}^*} a_n \left\vert v-v_*\right\vert^{2n}.
\end{align*}
It would be clearly feasible to handle such cross sections in the computations for $\Xi_i$ above -- see (18). 
However an inspection of the computations in the proof of Proposition \ref{prop:balance-laws} suggests that any such consideration of a general cross section would only complicate the computations, without giving a reasonable added value about the structure of the equation. For this reason, we have not considered here this possibility.
 
\subsection{Asymptotic analysis}\label{ssec:asymptotic}

We are now ready to consider, at the formal level the $\eps\to0$ limit of the system \eqref{eq:eps-mass-bal}-\eqref{eq:eps-momen-bal}-\eqref{eq:eps-energy-bal}.

In the following, let us set the fluxes for $i=1,\dots,n$:
\begin{align*}
J_i^\eps(t,x) := \frac{1}{\eps}\int\limits_{\R^3}v\, f_i^\eps(t,x,v)\,\xd v = c_i^{\eps}(t,x)u_i^{\eps}(t,x) \quad \mbox{ for } (t,x)\in(0,\infty)\times\R^3, 
\end{align*}
and denote, for any $t\ge0$ and $x\in\R^3$,
\begin{align*}
& c_i(t,x) := \lim_{\eps\to 0^+} c_i^{\eps}(t,x); & ~ & J_i(t,x) := \lim_{\eps\to0^+}J_i^\eps(t,x); 
\\[0.2 cm]
& T_i(t,x) := \lim_{\eps\to0^+}T_i^\eps(t,x); & ~ & c_{\rm tot}(t,x) := \lim_{\eps\to0^+}c^\eps_{\rm tot}(t,x),
\end{align*}
where 
$$
\displaystyle c^\eps_{\rm tot}(t,x) := \sum_{i=1}^n c^\eps_i(t,x).
$$
We first note that, at the leading order, all the kinetic temperatures of the components of the mixture converge to the same limit.

\begin{Lem}
Suppose that the distribution functions $f^\eps_i(t,x,v)$ preserve the local Maxwellian structure \eqref{eq:Maxwellian} for all time $t>0$. Then, we have
\begin{align*}
T^\eps_i (t,x) -T^\eps_j (t,x) = \mathcal{O}(\eps^2)
\qquad
\mbox{ for }(t,x)\in(0,\infty)\times\R^3,
\quad i,j=1,\dots,n.
\end{align*}
Consequently,
$$
\lim_{\eps\to0^+}T_i^\eps(t,x)= T(t,x)
$$
for all $i=1,\dots,n$.
\end{Lem}

\begin{proof}
By considering the leading order terms in the energy balance \eqref{eq:eps-energy-bal}-\eqref{eq:xi-eps} we have:
\begin{align*}
\sum_{j\not=i}
\|b_{ij}\|_{L^1}
\frac{m_j}{(m_i + m_j)^2}
c^\eps_i c^\eps_j \left(T^\eps_i-T^\eps_j\right)
= \mathcal{O}(\eps^2)
\quad 
\mbox{ for each }i=1,\dots,n.
\end{align*}
In the limit, the above relations are linearly dependent. Hence we deduce that
\begin{align*}
T_i (t,x) = T_j (t,x) \equiv T(t,x)
\qquad
\mbox{ for }(t,x)\in(0,\infty)\times\R^3
\end{align*}
for each $i,j=1,\dots,n$.

\end{proof}

Then, the following theorem holds.

\begin{Thm}\label{thm:formal-eps}
Let $(c_i, J_i, T)$ be the limit, as $\eps\to 0^+$ of the quantities $(c_i^\eps, J_i^\eps, T_i^\eps)$.
Then the macroscopic observables $(c_i, J_i, T)$ solve the system

\begin{equation}\label{eq:coupled-system-not-closed}
\left\{
\begin{array}{lll}
\displaystyle
\partial_t c_i + \nabla_x\cdot J_i = 0 &\qquad \mbox{\rm  on } (0,\infty)\times \R^3, &\quad i=1,\dots,n\\ \\
\displaystyle \nabla_x \left( c_i T \right) = -\sum_{j\neq i} \frac{c_j J_i-c_i J_j}{\dij} &\qquad
\mbox{\rm  on } (0,\infty)\times \R^3, &\quad i=1,\dots,n\\ \\
\displaystyle \partial_t \left( c_{\rm tot} T \right) + \frac{5}{3}\nabla_x \cdot \left( T \sum_{i=1}^n J_i \right) = 0  &\qquad
\mbox{\rm  on } (0,\infty)\times \R^3,\cr
\end{array}
\right.
\end{equation}
where $\displaystyle c_{\rm tot}(t,x) := \sum_{i=1}^n c_i(t,x)$ is the total concentration and the binary diffusion coefficients $\dij$ are given by
\begin{align*}
\dij = \frac{k}{2\pi \|b_{ij}\|_{L^1}} \frac{(m_i + m_j)}{m_i m_j}
\qquad i\neq j, \, \, i,j=1,\dots,n.
\end{align*}
Furthermore, the sum $\dsp\sum_{i=1}^n c_i(t,x) T(t,x)$ is space-independent.
\end{Thm}

\begin{proof}
The first equation of \eqref{eq:coupled-system-not-closed} can be straightforwardly obtained by performing the formal limit, as $\eps\to 0$, of Equation
\eqref{eq:eps-mass-bal}.
Deducing the flux-gradient relations is straightforward: equate the $\mathcal{O}(1)$ terms in the momentum balance equations \eqref{eq:eps-momen-bal}-\eqref{eq:theta-eps} and pass to the limit as $\eps\to 0$. Observe that the binary diffusion coefficients $\dij = \dji$, thanks to our earlier observation that $b_{ij} = b_{ji}$ for $i,j=1,\dots,n$ on the angular collision kernels. To show that the sum $\sum_{i=1}^n c_i T$ is space-independent, consider the flux-gradient relations on the second line of Equation \eqref{eq:coupled-system-not-closed}, sum over the index $1\le i\le n$ and pass to the limit. The result is
\begin{equation}
\label{eq:nabla-c-tot-T=0}
\nabla \left(\sum_{i=1}^n c_i T \right) = \nabla (  c_{\rm tot} T )= 0
\end{equation}
which in turn implies that
\begin{equation*}
\sum_{i=1}^n c_i(t,x) T(t,x) = g(t)
\qquad
\mbox{ for }(t,x)\in(0,\infty)\times\R^3,
\end{equation*}
for some function $g(t)$.
At the next order in the energy balance, by summing Equations \eqref{eq:eps-energy-bal}-\eqref{eq:xi-eps} over the index $1\le i\le n$, we have that the right-hand side is identically zero by symmetry and hence, at the leading order in $\eps$, we obtain the transport equation:
$$
\partial_t \left( c_{\rm tot} T \right) + \frac{5}{3}\nabla_x \cdot \left( T \sum_{i=1}^n J_i \right) = 0.
$$
\end{proof}

\begin{Rem}\label{rem:bisi-etal}
In our setting, all the kinetic temperatures of the species tend to the same limit $T$. In the literature, some authors considered limit procedures leading to multi-temperature and multi-velocity fluid-dynamic models, starting from the rescaled system
\begin{align*}
\partial_t f^\eps_i + v\cdot \nabla f^\eps_i = \frac{1}{\eps} Q_{ii}\left(f^\eps_i, f^\eps_i\right) + \sum_{j\neq i} Q_{ij}(f_i^{\eps},f_j^{\eps})
\end{align*}
for $i\in\{1,\dots,n\}$, which can be interpreted to be under the hyperbolic time-space scaling, i.e., $(t,x)\mapsto(\eps t, \eps x)$.
More details can be found in \cite{Bisi_2011, Bisi_2012, Pavic_2014, simic2015non}.
\end{Rem}

\subsection{Closure relation}\label{ssec:closure}

Note that the flux-gradient relations, i.e., the second line of equations in \eqref{eq:coupled-system-not-closed} are linearly dependent. So, we have $2n$ independent equations for $2n+1$ unknowns $(c_i, J_i, T)$. This necessitates an additional equation -- a closure relation. 
Of course, the correct closure depends on the physical setting and on the observed physical phenomena. We propose here, as an example, a possible closure relation, which basically imposes that, as a whole, the total concentration follows a Fickian behaviour, and analyse its consequences.

We hence suppose that the sum of the molar fluxes is locally proportional to the gradient of the total molar concentration, i.e.,
\begin{align}\label{eq:closure-relation-I}
\sum_{i=1}^n J_i = - \alpha \nabla c_{\rm tot} =  \alpha c_{\rm tot} \frac{\nabla T}{T}
\end{align}
for some proportionality constant $\alpha>0$, where the last equality comes from Equation \eqref{eq:nabla-c-tot-T=0}.
The negative sign guarantees that the mass flows locally in the opposite direction of the total concentration gradient. We name this closure the \textsl{decoupling closure relation}, as this leads to decoupled equations for the total concentration $c_{\rm tot}$ and the temperature field $T$.

\begin{Rem}\label{rem:closure-classical}
In \eqref{eq:nabla-c-tot-T=0}, if we take the temperature field to be constant (i.e., isothermal case), then $\nabla c_{\rm tot} = 0$. Hence the closure relation \eqref{eq:closure-relation-I} would yield in this scenario:
\[
\sum_{i=1}^n J_i = 0,
\]
i.e., the sum of the molar fluxes is locally identically zero. This is indeed the classical closure equation for the Maxwell-Stefan diffusion model in the isothermal setting.
\end{Rem}

\subsection{Qualitative properties of the coupled system for $c_{\rm tot}$ and $T$}\label{ssec:quality}
Associated with the closure relation \eqref{eq:closure-relation-I}, we shall derive a system of equations for the total molar concentration and for the temperature field.
\begin{Lem}\label{lem:system-closure-I}
Under the closure relation \eqref{eq:closure-relation-I}, the unknowns $c_{\rm tot}$ and $T$ satisfy the equations:
\begin{align}
& \partial_t c_{\rm tot} - \alpha \Delta c_{\rm tot} = 0,\label{eq:c-tot-closure-I}
\\[0.2 cm]
& \partial_t T - \left( \frac{2}{3} \partial_t \log c_{\rm tot} \right) T - \left( \frac{5\alpha}{3} \nabla \log c_{\rm tot} \right)\cdot \nabla T = 0.\label{eq:T-closure-I}
\end{align}
\end{Lem}

\begin{proof}
Sum the mass balance equations in \eqref{eq:coupled-system-not-closed} over the index $i=1,\dots,n$ yielding
\begin{align*}
\partial_t c_{\rm tot} + \nabla \cdot \left( \sum_{i=1}^n J_i \right) = 0.
\end{align*}
Substituting the closure relation \eqref{eq:closure-relation-I} in the above equation yields the parabolic equation \eqref{eq:c-tot-closure-I}.\\
Next we substitute the closure relation \eqref{eq:closure-relation-I} in the energy balance equation in \eqref{eq:coupled-system-not-closed} which yields
\begin{align*}
c_{\rm tot} \partial_t T + T \partial_t c_{\rm tot} - \frac{5\alpha}{3} \nabla c_{\rm tot} \cdot \nabla T - \frac{5\alpha}{3} T \Delta c_{\rm tot} = 0.
\end{align*}
Substituting for $\Delta c_{\rm tot}$ using \eqref{eq:c-tot-closure-I} in the above equation yields the advection equation \eqref{eq:T-closure-I} for the temperature field $T(t,x)$.
\end{proof}

Note that the system of equations \eqref{eq:c-tot-closure-I}-\eqref{eq:T-closure-I} are decoupled. One can solve the heat equation \eqref{eq:c-tot-closure-I} for the total concentration $c_{\rm tot}$ and treat it as a known coefficient in the advection problem \eqref{eq:T-closure-I} for $T(t,x)$. 

For system \eqref{eq:c-tot-closure-I}-\eqref{eq:T-closure-I}, we record a maximum principle.

\begin{Prop}\label{prop:max-princ-cT}
Suppose the initial data $c_{\rm tot}^{\rm in}$ and $T^{\rm in}$ to the evolution equations \eqref{eq:c-tot-closure-I}-\eqref{eq:T-closure-I} are non-negative and satisfy
\begin{align*}
0 < c_{\rm min} \le c_{\rm tot}^{\rm in}(x) \le c_{\rm max} <\infty; \quad
0 < T_{\rm min} \le T^{\rm in}(x) \le T_{\rm max} <\infty.
\end{align*}
Then
\begin{align*}
c_{\rm min} \le c_{\rm tot}(t,x) \le c_{\rm max} \quad \mbox{ for }(t,x)\in [0,\infty)\times\R^3.
\end{align*}
Furthermore
\begin{align}\label{eq:temp-solution}
T(t,x) = T^{\rm in}({\scriptstyle X}(0;t,x)) e^{\frac23\int\limits_0^t \partial_t \left(\log c_{\rm tot}\right) (s,{\scriptscriptstyle X}(s;t,x))\, {\rm d}s} \quad \mbox{ for }(t,x)\in [0,\infty)\times\R^3,
\end{align}
where ${\scriptstyle X}(s;t,x)$ is the flow associated with the vector field $\mathcal{V}(t,x):= -\frac{5\alpha}{3}\nabla \log c_{\rm tot}$.
\end{Prop}
\begin{proof}
Standard maximum principles on the heat equation implies that the solution to \eqref{eq:c-tot-closure-I} stays non-negative and is bounded both from above and below, the same as the initial data.\\
To solve for the temperature field $T(t,x)$ in the evolution equation \eqref{eq:T-closure-I}, consider the solution ${\scriptstyle X}(s;t,x)$ to the differential equation
\begin{align*}
\frac{{\rm d}{\scriptstyle X}}{{\rm d}s}(s;t.x) & = \mathcal{V}(s, {\scriptstyle X}(s;t,x))
\\
{\scriptstyle X}(t;t.x) & = x.
\end{align*}
Equipped with the flow ${\scriptstyle X}(s;t.x)$, the solution $T(t,x)$ as given by \eqref{eq:temp-solution} follows.
\end{proof}

This article concerns a formal derivation of the diffusion limit for the gas mixtures in a non-isothermal setting. Proving an existence-uniqueness result for the proposed system is out of the scope of this present article. The authors are currently working on this problem and expect to have a publication in the near future.

\medskip
\medskip

\noindent {\bf Acknowledgments:} This work was partially funded by the projects \textit{Kimega} (ANR-14-ACHN-0030-01) and
\textit{Kibord} (ANR-13-BS01-0004). H.H. acknowledges the support of the ERC grant MATKIT and the EPSRC programme grant ``Mathematical fundamentals of Metamaterials for multiscale Physics and Mechanic'' (EP/L024926/1).

\end{document}